\documentclass[12pt]{article}
\usepackage{amsmath}
\usepackage{amsfonts}
\usepackage{amssymb}
\usepackage{latexsym}
\usepackage{graphicx}
\usepackage{subfig}
\usepackage{enumerate}
\usepackage[shortlabels]{enumitem}
\usepackage{mathtools}

\usepackage{amsmath, amssymb,latexsym}
\usepackage{color}
\usepackage{xcolor}
\usepackage{fullpage}
\usepackage{hyperref}

\usepackage[T1]{fontenc}
\usepackage[utf8]{inputenc}
\usepackage{lmodern}

\def\N{\mathbb{N}}

\newcommand{\R}{\mathbb R}
\newcommand{\C}{\mathbb C}
\newcommand{\F}{\mathbb F}
\newcommand{\E}{\mathcal E}

\newcommand{\PP}{\mathcal M}
\newcommand{\dR}{d_{\R}}
\newcommand{\dC}{d_{\C}}

\DeclareMathOperator{\spann}{span}
\newcommand{\ip}[2]{\langle #1,#2\rangle}

\definecolor{brilliantrose}{rgb}{1.0, 0.33, 0.64}
\definecolor{myviolet}{rgb}{0.21, 0.0, 0.85}
\definecolor{amethyst}{rgb}{0.6, 0.4, 0.8}
\definecolor{carrotorange}{rgb}{0.93, 0.57, 0.13}

\DeclareMathOperator{\rank}{rank}

\DeclareMathOperator{\tr}{trace}

\newtheorem{theorem}{Theorem}[section]
\newtheorem{proposition}[theorem]{Proposition}

\newtheorem{lemma}[theorem]{Lemma}

\newtheorem{corollary}[theorem]{Corollary}
\newtheorem{remark}[theorem]{{\sc Remark}}

\title{Matrix nearness problems:\\ Do real inputs admit real solutions?}
\author{Vanni Noferini\thanks{Corresponding author. Supported by a Research Council of Finland grant (decision number
370932). Address: Department of Mathematics and Systems Analysis, Aalto University, P.O. Box 11100, FI-00076, Aalto, Finland. Email: vanni.noferini@aalto.fi} \ and \ Matvei Zhukov\thanks{Supported by a Research Council of Finland grant (decision number
370932). Address: Department of Mathematics and Systems Analysis, Aalto University, P.O. Box 11100, FI-00076, Aalto, Finland. Email: matvei.zhukov@aalto.fi}}
\date{\today}

\begin{document}
\maketitle

\begin{abstract}
Given a real square matrix $A$ and a nonempty closed target set of complex matrices $\E$, invariant by complex conjugation and containing real matrices, does the subset of $\E$ consisting of the matrices nearest to $A$ in the Frobenius distance always contain a real matrix?
We give negative answers when $\E$ is either the set of normal matrices (solving a question posed by N. Higham) or the set of matrices whose eigenvalues lie in the closed left half-plane (solving a question posed by the first author and F. Poloni).  For the latter problem, we also argue that the ratio between the real and complex distances is unbounded in every dimension $n\geq 3$, and this holds for every distance induced by a unitarily invariant norm. For both problems and $n \geq 3$, we show that a real input $A$ uniformly drawn from the unit sphere $\|A\|_F=1$ has no real minimizer over $\E$ with probability strictly between $0$ and $1$. Beyond these two cases, we prove for an arbitrary nonempty closed target set $\E$ that, for almost every real input, either $\E$ contains a unique nearest matrix which is real, or every nearest matrix in $\E$ is nonreal.
\end{abstract}


\section{Introduction}\label{sec:intro}

\subsection{Motivation}

In the beautiful expository paper \cite{Waterhouse}, W.~C. Waterhouse demolished the common (but wrong) intuition that, if a multivariate real optimization problem enjoys a particular symmetry in its variables, then a global minimizer can be always chosen to have the same symmetry. Waterhouse called this intuition the
Purkiss Principle, and demonstrated that it is generally false. In this paper, we similarly tackle a matrix version of the Purkiss Principle in the context of matrix nearness problems \cite{GLbook,Higham,NP2}. In this setting, the symmetry that concerns us is complex conjugation: If the input data are real and the target set is invariant by conjugation and contains real matrices, must there be a real solution?

More formally, let us consider a real matrix $A\in\R^{n\times n}$ and a closed target set $\E\subseteq\C^{n\times n}$ such that
\begin{equation}\label{eq:Eprop}
    X\in\E\ \Longrightarrow\ \overline X\in\E,
 \qquad \E\cap\R^{n\times n}\ne\varnothing. 
\end{equation}
For most (but not all) results in this paper, distances are measured in the Frobenius norm, denoted by $\| \cdot \|_F$. We write
\begin{equation}\label{eq:distances}
 \dC(A)=\min_{X\in\E}\|A-X\|_F,
 \qquad
 \dR(A)=\min_{X\in\E\cap\R^{n\times n}}\|A-X\|_F.
\end{equation}
Both minima exist, because both $\mathcal{E}$ and its intersection with $\R^{n \times n}$ are by assumption nonempty closed subsets of a finite dimensional metric space. It is also clear by their definition that $\dC(A)\le\dR(A)$. We ask for which target sets $\mathcal{E}$ equality holds for \emph{every} real input $A$. For some of our results, we will generalize the Frobenius distance by that induced by an arbitrary unitarily invariant norm.

Intuitively, one might expect a positive answer to our question. If $B \in \E$ is nearest to
a real $A$, then so is $\overline B \in \E$, using  \eqref{eq:Eprop}. In particular, whenever the matrix in $\E$ nearest to $A$ is unique, then it must be real. This happens, for example, when $\E$ is convex. Even for a nonconvex target $\E$, as long as the property $\Re\E\subseteq\E$ holds then $B \in \E$ implies $\Re B\in\E$, and therefore
\begin{equation}\label{eq:pythagoras}
 \|A-B\|_F^2=\|A-\Re B\|_F^2+\|\Im B\|_F^2,
\end{equation}
forcing every nearest matrix to be real. One example of a nonconvex set satisfying $\Re \E \subseteq \E$ is the set of $n \times n$ matrices with at most $m$ nonzero elements. Indeed, taking the real part cannot create further nonzeros, and on the other hand, for example when $m=1,n=2$, we have
\[  \begin{bmatrix}
    1&0\\
    0&0
\end{bmatrix} \in \E, \qquad \begin{bmatrix}
    0&0\\
    0&1
\end{bmatrix} \in \E, \qquad \begin{bmatrix}
    \frac12 & 0\\
    0 & \frac12
\end{bmatrix}  \not\in \E, \]
showing that $\E$ is nonconvex. (More generally, $\E$ is not convex if $0<m<n^2$.) A much more classic example of a nonconvex target set for which the answer is positive comes from the Eckart-Young-Mirsky-Schmidt Theorem \cite{EckartYoung,Mirsky, Schmidt}: There is always a real best approximation of prescribed maximum rank to a real input  even when the optimization problem is posed over complex matrices. The result can be further generalized. Indeed, it was proved in \cite{NP2} that every real input $A$ admits a real closest matrix whenever the target set $\E$ is unitarily invariant, that is, if $U,V$ are unitary matrices then $X \in \E \Rightarrow U^* X V \in \E$. Besides low rank approximations, this family of target sets includes for example $\E=U(n)$ (nearest unitary matrix); more examples are given in \cite{NP2}.

On the other hand, it is not hard to see that the properties \eqref{eq:Eprop} alone are insufficient. Even for the scalar case $n=1$, one can for instance take $\E=\{2,i,-i\}$ and $A=0$, and observe that $\dC(A)=1<2=\dR(A)$, or in other words the minimizer breaks the problem's symmetry. A question of particular interest is whether the same phenomenon occurs for sets $\mathcal{E}$ that arise in matrix nearness problems that have practical applications. Matrix nearness problems have a long history in numerical linear algebra \cite{Higham} and they still are an extremely active area \cite{GLbook,NP2}. In the classic survey \cite{Higham}, N. Higham asked whether a real matrix always has a real nearest normal matrix when complex competitors are allowed; more recently, some numerical evidence possibly suggesting a negative answer was given in \cite[Section~4]{NZ}. For the analogous question in the context of the nearest Hurwitz stable matrix problem, a positive answer was explicitly conjectured by V. Noferini and F. Poloni \cite[Conjecture 1]{NP}.

Moving back to general matrix nearness problems, it is worth mentioning that there is also a numerical motivation to study the question that we analyze in this manuscript. In nearness problems arising from applications where the input data represent a real system, complex perturbations may not be meaningful; on the other hand, algorithms working over complex matrices can in practice be easier to implement, or better performing (or both), than algorithms preserving the real structure. Thus, the question of whether optimization over complex matrices yields the same distance as optimization over real matrices has concrete implications that go beyond the theoretical interest of the question; indeed, for target sets and inputs for which the answer is negative, the motivation to develop real algorithms is obviously strengthened.

In this paper, we first prove a general result for the Frobenius distance and an arbitrary nonempty closed target set: For almost every real input, either there is a unique nearest matrix and it is real, or all nearest matrices are nonreal. This in turn implies that, for such target sets that are also invariant by complex conjugation and contain real matrices, requiring that every real input has a real minimizer is equivalent to asking that \emph{almost} every real input has a unique minimizer. We then turn our attention to the two concrete instances of nearness problems mentioned above, i.e., the nearest normal matrix problem and the nearest Hurwitz stable matrix problem. We show that, for both these matrix nearness problems, the question in our paper's title has a negative answer when the matrix size $n \geq 3$. We provide counterexamples already for $n=3$. For $3 \times 3$ normal matrices, we first compute the real distance by an optimization over the unit sphere in $\R^3$. This then allows us to construct a counterexample by exhibiting a real input admitting a nonreal normal matrix at a smaller distance. For stable matrices, we argue that the failure is even more pronounced in the sense that allowing complex entries can improve the distance by an arbitrarily large relative factor, and we exhibit a family of examples to demonstrate this behavior. For both problems, these observations remain true in every dimension $n\geq3$, and for the case of stable matrices it also remains true for every distance induced by a unitarily invariant norm. On the contrary, we show that for the Frobenius distance this phenomenon does not occur when $n\leq2$. Moreover, for $n \geq 3$ we prove that there exists an open neighborhood of inputs in which there is a strict gap between $\dR$ and $\dC$, thus showing that the counterexamples are not exceptional inputs at which an otherwise valid principle fails, and that indeed a similar behavior can be expected with positive probability. On the other hand, we also exhibit nonempty open sets of inputs that provably admit a real nearest (either normal or Hurwitz stable) matrix. Hence, when $n \geq 3$, for both problems and for an input sampled uniformly from the unit sphere, the probability of having a real minimizer is strictly between $0$ and $1$. For stable matrices, this remains true even if we restrict ourselves to unstable inputs.

The structure of the paper is the following. In Section~\ref{sec:general} we establish general facts valid for an arbitrary nonempty closed target set $\mathcal{E} \subseteq \C^{n \times n}$, as well as some further properties valid for target sets that additionally satisfy \eqref{eq:Eprop}. We then discuss more closely the cases, especially relevant for applications \cite{GLbook,Higham,NP,NP2,NZ,Ruhe}, of normal and Hurwitz stable matrices in Sections~\ref{sec:normal} and~\ref{sec:stable}, respectively. To keep those sections focused on the mathematical proofs, we postpone to Section \ref{sec:method} an intuitive account of how the basic $3 \times 3$ counterexamples were constructed.

\subsection{Notation}

Given a matrix $X \in \C^{n \times n}$, $\overline{X}$ denotes the elementwise conjugate of $X$, $X^*=(\overline{X})^T$ denotes the conjugate transpose of $X$,  whereas $\Re X$ and
$\Im X$ are (resp.) the elementwise real and imaginary parts of $X$. The Frobenius inner product
on $n \times n$ complex matrices (seen as a real Hilbert space of dimension $2n^2$) is
$\ip{X}{Y}=\Re\tr(Y^*X)$; this inner product induces the Frobenius norm, defined as $\| X \|_F=\sqrt{\ip{X}{X}}$, and the Frobenius distance between two matrices $X$ and $Y$, defined as $\|X-Y\|_F$. Throughout the paper, ``almost all'' and ``almost everywhere'' statements about real inputs are relative
to the Lebesgue measure on $\R^{n\times n}$.

\section{General target sets}\label{sec:general}

In this section, $\E$ denotes a nonempty closed subset of $\C^{n \times n}$.
Having fixed one such target set $\E$, we define a subset $\PP(A) \subseteq \E$ (depending also on an input $A$) and a function $f$ as follows:
\begin{equation}\label{eq:Pandf}
  \PP(A):=\{B\in\E:\|A-B\|_F=\dC(A)\},
 \quad f:\R^{n\times n} \rightarrow [0,+\infty[, \ A \mapsto f(A) = \dC(A)^2.   
\end{equation}
 In \eqref{eq:Pandf}, the notation $\dC(A)$ refers to \eqref{eq:distances}. We also remark that $f$ in \eqref{eq:Pandf} is a real function of a real matrix input. Below, we will discuss its derivatives, and these should be interpreted accordingly with respect to a real matrix variable. The other quantity defined in \eqref{eq:Pandf} is the set $\PP(A)$ of the minimizers in $\E$ of the distance from $A$. It is clear that $\PP(A)$ is nonempty, because $\mathcal{E}$ is closed and nonempty, and we show in Lemma \ref{lem:compact} that it is compact.

 \begin{lemma}\label{lem:compact}
     For every nonempty closed set $\E \subseteq \C^{n \times n}$ and every $A \in \R^{n \times n}$, the set $\PP(A)$ defined in \eqref{eq:Pandf} is compact.
 \end{lemma}
 \begin{proof}
     The set $\PP(A)$ is the intersection of $\E$, which is closed by assumption, and the preimage of the closed set $\{\dC(A)\}$ under the continuous function $X \mapsto \|A-X\|_F$. Thus, $\PP(A)$ is closed. Moreover, $\PP(A)$ is bounded because, for every $B \in \PP(A)$, the triangle inequality yields $\|B\|_F \leq \|A\|_F + \dC(A)$.
 \end{proof}

 Proposition \ref{prop:differentiability} below describes some further useful properties of the elements of $\PP(A)$.

\begin{proposition}\label{prop:differentiability}
Fix a nonempty closed set $\E \subseteq \C^{n \times n}$, and let $\PP(A)$ and $f$ be defined as in \eqref{eq:Pandf}. Then, the function $f$ is Fr\'{e}chet-differentiable at $A\in\R^{n\times n}$ if and only if all matrices in $\PP(A)$ have the same real part. Moreover, at every point of differentiability of $f$, it holds
\begin{equation}\label{eq:gradient}
 \nabla f(A)=2(A-\Re B)\qquad \forall \ B\in\PP(A),
\end{equation}
and the imaginary parts of all the matrices in $\PP(A)$ have the same Frobenius norm.
\end{proposition}

\begin{proof}
Suppose first that $f$ is Fr\'{e}chet-differentiable at $A$, and fix $B\in\PP(A)$. The
function
\[
 g(X)=\|X-B\|_F^2-f(X),\qquad X\in\R^{n\times n},
\]
is everywhere nonnegative and vanishes at $A$. Thus its gradient at $A$ (which exists because $g$ is the sum of two functions both differentiable at $A$) is zero, yielding  \eqref{eq:gradient}. The final part of the statement about imaginary parts follows from
\eqref{eq:pythagoras}.

We next prove the converse implication of the first part of the statement. Suppose that there exists a real matrix $R\in\R^{n\times n}$ such that $\Re B=R$ for every $B\in\PP(A)$.
Let $(H_k)_k\subset\R^{n\times n}$ be an arbitrary sequence with $H_k\ne0$ and $H_k\to0$, and construct a second sequence $(B_k)_k$ by choosing $B_k\in\PP(A+H_k)$ for every $k$. We claim that $\Re B_k\to R$, independently of our choice of the two sequences.

Assuming the claim, let us fix $B\in\PP(A)$, and observe that (using $\Re B=R$) 
\[ f(A+H_k) \leq \|A+H_k-B \|_F^2 = f(A) + 2\ip{A-R}{H_k} +\|H_k\|_F^2.  \]
On the other hand, since $f(A+H_k) = \|A+H_k-B_k \|_F^2$ and expanding the right hand side, 
\[ f(A+H_k) = \|A-B_k \|_F^2 + 2 \ip{A-\Re B_k}{H_k} + \|H_k\|_F^2 \geq f(A) + 2 \ip{A-\Re B_k}{H_k} + \|H_k\|_F^2.  \]
Thus,
\[ \|H_k\|_F^2 + 2\ip{R-\Re B_k}{H_k}  \leq f(A+H_k)-f(A)-2\ip{A-R}{H_k} \leq \|H_k\|_F^2  \]
which implies, using the Cauchy-Schwarz inequality,
\[ \left|f(A+H_k)-f(A)-2\ip{A-R}{H_k}\right|  \leq \|H_k\|_F^2 +2\|\Re B_k-R\|_F\,\|H_k\|_F  =o(\|H_k\|_F).\]
Since $(H_k)_k$ was an arbitrary sequence, this last inequality proves that $f$ is Fr\'{e}chet-differentiable at $A$ and $\nabla f  =2(A-R)$.

It only remains to prove the claim. To this goal, suppose for a contradiction that $\Re B_k$ does not converge to $R$. Then, there exists $\varepsilon>0$ and a subsequence $(B_{k_j})_j$ such that, for all $j$, $\|\Re B_{k_j}-R\|_F\ge\varepsilon$. On the other hand, $(B_k)_k$ is bounded because, by the triangle inequality,
\[  \|B_k\|_F \leq \|A+H_k\|_F+\dC(A+H_k)\leq \|A\|_F+\dC(A)+2\|H_k\|_F. \]
A fortiori, the subsequence $B_{k_j}$ is also bounded. Therefore, it admits a convergent sub-subsequence $(B_{k_{j_i}})_i$; call $B_\infty$ the limit of such sub-subsequence. Since $\E$ is closed and $(B_k)_k \subset \E$, we have that $B_\infty \in\E$. Moreover, the continuity of the Frobenius distance yields
\[ \|A-B_\infty\|_F  =\lim_{i\to\infty}\|A+H_{k_{j_i}}-B_{k_{j_i}}\|_F  =\lim_{i\to\infty}\dC(A+H_{k_{j_i}})  =\dC(A). \]
We conclude that $B_\infty\in\PP(A)$, and hence $\Re B_\infty=R$. This is a contradiction and it proves the claim that $\Re B_k \to R$.
\end{proof}

As a next step, in Corollary \ref{cor:equivalence} we recast our question about whether every real input has a real minimizer in terms of having or not a unique minimizer for \emph{almost} every real input. To achieve this goal, we first show in Theorem \ref{thm:dichotomy} that generically there are only two possible situations for a real input $A$, namely, either there is a unique nearest matrix in $\E$ and it is real, or there is no real nearest matrix at all in $\E$.

\begin{theorem}\label{thm:dichotomy}
Let $\E \subseteq \C^{n \times n}$ be a nonempty closed  set. For almost every $A\in\R^{n\times n}$, precisely one of the following
statements holds:
\begin{enumerate}
 \item There is a unique matrix in $\E$ nearest to $A$, and it is real.
 \item All the matrices in $\E$ nearest to $A$ are nonreal.
\end{enumerate}
If, in addition, $\E$ satisfies \eqref{eq:Eprop}, then in the second case there are in $\E$ at least two distinct matrices nearest to $A$.
\end{theorem}

\begin{proof}
The distance function $\dC$ is Lipschitz, and thus $f$ in \eqref{eq:Pandf} is locally Lipschitz. Therefore, by Rademacher's theorem \cite[Section~3.1]{EvansGariepy}, $f$ is  almost everywhere differentiable. Let $A$ be a point at which $f$ is differentiable. If there is a real matrix $C\in\PP(A)$, then by Proposition~\ref{prop:differentiability} $\Re B=C$ for all $B\in\PP(A)$. Hence,
\[  \|A-C\|_F^2=\|A-B\|_F^2  =\|A-C\|_F^2+\|\Im B\|_F^2, \]
and therefore $B=C$. Otherwise, all nearest matrices in $\PP(A)$ are nonreal by assumption. In this latter scenario, suppose further that $\E$ satisfies \eqref{eq:Eprop}; then, $B \in \PP(A) \Leftrightarrow \overline{B} \in \PP(A)$ and $B \neq \overline{B}$, showing that $\PP(A)$ contains at least two elements.
\end{proof}

\begin{corollary}\label{cor:equivalence}
Let $\E \subseteq \C^{n \times n}$ be a closed set satisfying \eqref{eq:Eprop}. Then, every real input admits a real minimizer of its distance to $\E$ if and only if almost every real input has a unique minimizer over $\E$.
\end{corollary}

\begin{proof}
\begin{itemize}
    \item[$\Rightarrow$] If every real input has a real minimizer, then case~\textup{2.} of Theorem~\ref{thm:dichotomy} never happens, and therefore case~\textup{1.} occurs for almost every $A$.
\item[$\Leftarrow$] If the minimizer is unique, then it is real by conjugation invariance. Hence, $\dC=\dR$ almost everywhere. But both distance functions are continuous, and therefore equality holds everywhere.
\end{itemize}
\end{proof}

We next prove Theorem \ref{thm:robustness}, which shows that, if the set of real inputs not admitting a real nearest matrix is nonempty, then it is open (and thus it has in particular positive Lebesgue measure in $\R^{n \times n}$).

\begin{theorem}\label{thm:robustness}
Let $\E \subseteq \C^{n \times n}$ be a closed set satisfying $\E \cap \R^{n \times n} \neq \varnothing$, and define $d_\C$ and $d_\R$ as in \eqref{eq:distances}. Then, the set
\[  \mathcal G=\{A\in\R^{n\times n}:\dC(A)<\dR(A)\} \]
is open. In particular, if $d_\C(A) < d_\R(A)$, then $A+E \in \mathcal{G}$ for every $E \in \R^{n \times n}$ such that \[ \|E\|_F< \frac{ d_\R(A)-d_\C(A)}{2}.\]
\end{theorem}

\begin{proof}
By the triangle inequality and for every  $E \in \R^{n \times n}$, 
\[ d_\R(A) \leq d_\R(A+E) + \| E \|_F \qquad \text{and} \qquad d_\C(A+E) \leq d_\C(A) + \|E\|_F.\]
Suppose now that $A \in \mathcal{G}$ and $2 \|E\|_F < d_\R(A)-d_\C(A)$. Then, using also the above inequalities,
\[ d_\R(A+E) - d_\C(A+E) \geq d_\R(A)-d_\C(A) - 2 \|E\|_F > 0.    \]
\end{proof}

\begin{remark}
    If, in Theorem \ref{thm:robustness}, the Frobenius distances $d_\R$ and $d_\C$ are replaced by distances induced by an arbitrary matrix norm $\nu$, and the Frobenius norm $\|E\|_F$ is replaced by $\nu(E)$, the statement still holds (with the same proof).
\end{remark}

Several target sets of practical interest in applications are invariant under scaling by a positive real number, i.e., if $X \in \mathcal{E}$ and $t > 0$ then $t X \in \mathcal{E}$. We express succinctly this property by writing $t \mathcal{E}=\mathcal{E}$ for all $t > 0$. This is true in particular of the two sets that we will more deeply analyze in Section \ref{sec:normal} and Section \ref{sec:stable}, that is, normal matrices and Hurwitz stable matrices, but also of other classic examples in matrix nearness problems (say, matrices having rank at most $r$). Corollary \ref{thm:cone} discusses an additional property that can be established in this case.

\begin{corollary}\label{thm:cone}
Let $\E \subset \C^{n \times n}$ be a closed set satisfying both \eqref{eq:Eprop} and $t\E=\E$ for all $t > 0$. Then, if $\mathcal{G}$ is defined as in Theorem \ref{thm:robustness}, it holds $t\mathcal G=\mathcal G$ for all $t > 0$. Moreover, either $\mathcal G$ is empty or  $\mathcal{G} \cap \{ \|A\|_F=1\}$ has finite positive measure.
\end{corollary}

\begin{proof}
    Clearly, 
$ d_\R(tA)- d_\C(tA)= t \left( d_\R(A)-d_\C(A) \right)$, proving $\mathcal{G}=t\mathcal{G}$. The intersection of $\mathcal{G}$ with the unit sphere is relatively open, because $\mathcal{G}$ is open in $\R^{n \times n}$ by Theorem \ref{thm:robustness}. Moreover, if $\mathcal{G}$ is not empty, its intersection is not empty because for every $G \in \mathcal{G}$ we have that $G/\|G\|_F$ belongs to the intersection. The statement follows because a nonempty open subset of the sphere has positive and finite measure.
\end{proof}

Corollary \ref{thm:cone} applies in particular to the cases where $\E$ is either the set of normal matrices or the set of Hurwitz stable matrices. Below, we will construct in Theorem \ref{thm:normalexample} (resp. Theorem \ref{thm:stable}), examples of real inputs whose nearest normal (resp. stable) matrices are nonreal. Combining these results with Theorem \ref{thm:robustness} and Corollary \ref{thm:cone} leads us to conclude that open cones of such inputs exist, and that if one uniformly samples the unit sphere then there is a positive probability to encounter one such input.

\section{Nearest normal matrices}\label{sec:normal}

In this section, the target set $\E$ is the set of normal matrices, i.e., $\E=\{N\in\C^{n\times n}:NN^*=N^*N\}$. It is clear that this set is closed, is invariant by complex conjugation, and contains real matrices. The problem of finding a nearest normal matrix is classic \cite{Higham, Ruhe} and still an active area \cite{NP2,NZ}.

We begin our analysis by deriving a formula for the distance from a real $A$ to real normal matrices valid in dimension $n=3$. To this goal, we identify the set of real $3 \times 3$ skew-symmetric matrices with $\R^3$, via the bijection
\begin{equation}\label{eq:K}
     k=\begin{bmatrix}
    k_1\\
    k_2\\
    k_3
\end{bmatrix} \in \R^3 \mapsto \mathcal{K}(k) =  \begin{bmatrix}
 0&k_3&-k_2\\-k_3&0&k_1\\k_2&-k_1&0
 \end{bmatrix}. 
\end{equation}
Observe in particular that $\|\mathcal{K}(k)\|_F^2=2\|k\|_2^2$.

\begin{lemma}\label{lem:commutant}
Let $0\ne K=-K^T\in\R^{3\times3}$, and let $q \in \R^3$ be a unit vector in $\ker K$. The set of the real symmetric matrices that commute with $K$ is precisely
\[ \{X \in \R^{3 \times 3} : X=\alpha I_3+\beta qq^T \ \mathrm{for} \ \mathrm{some} \ \alpha,\beta\in\R\}.\]
\end{lemma}

\begin{proof}
By computing the real Schur form, there exists an orthogonal matrix $Q$ such that $Qe_3=q$ and $Q^TKQ=\lambda(J\oplus0)$ for some $\lambda\neq0$ and having defined $J:=\begin{bmatrix}
    0&1\\
    -1&0
\end{bmatrix}$.
A generic $3 \times 3$ symmetric block matrix $\begin{bmatrix}
    X & x\\
    x^T & \beta
\end{bmatrix}$, where $X=X^T \in \R^{2 \times 2}$, commutes with $J\oplus0$ precisely when $JX=XJ$ and $Jx=0$. These equations in turn imply that $X=\alpha I_2$ (for some $\alpha \in \R$) and $x=0$. It follows that a real symmetric matrix commutes with $K$ if and only if it has the form $\alpha I_3+(\beta-\alpha)qq^T$.
\end{proof}

\begin{corollary}\label{cor:commutant}
    Every real nonsymmetric normal  $3\times3$ matrix can be written as 
    \begin{equation}\label{eq:normalparam}
 N=\alpha I_3+\beta qq^T+\gamma \mathcal{K}(q),
 \qquad \|q\|_2=1,\quad \alpha,\beta \in \R, \qquad 0\neq\gamma\in\R.
\end{equation}
\end{corollary}
\begin{proof}
A real matrix $N=S+K$, with $S=S^T$ and $K=-K^T$, is normal if and only if $SK=KS$ \cite{GJSW}. Observe that, if $K\ne0$ and $\ker K = \spann(q)$, then $K=\gamma \mathcal{K}(q)$ for some $\gamma\ne0$ and where $\mathcal{K}$ is the map \eqref{eq:K}. The parametrization \eqref{eq:normalparam} then follows  by Lemma \ref{lem:commutant}.
\end{proof}

\begin{proposition}\label{prop:normaldistance}
Let $A=S+\mathcal{K}(k)\in\R^{3\times3}$, where $S=S^T$ and $\mathcal{K}$ is the map \eqref{eq:K}, and define
\[ \mu:=\tr(S)/3, \qquad T:=S-\mu I_3, \qquad \xi =  \max_{\|q\|_2=1} \left[\frac32(q^TTq)^2+2(k^Tq)^2\right]. \]
 Then, for the Frobenius distance from $A$ to real normal matrices it holds
\[ d_{\R}(A)^2 =2\|k\|_2^2 + \min\left\{0, \ \|T\|_F^2-\xi\right\}.\]
\end{proposition}

\begin{proof}
Every real $3 \times 3$ normal matrix, including any one closest to $A$, either has three (counted with multiplicity) real eigenvalues, and therefore is symmetric, or has one real eigenvalue and two nonreal eigenvalues, and therefore is not symmetric. Using also Corollary \ref{cor:commutant}, we can write the set of $3 \times 3$ real normal matrices as $\E \cap \R^{3 \times 3} = \E_0 \cup \E_1$, where $\E_0$ is the set of real symmetric matrices and $\E_1$ is the closure of the set of matrices of the form \eqref{eq:normalparam}. Hence,
\begin{equation}\label{eq:notapartition}
     \dR(A)^2 = \min \left\{  \min_{Y \in \E_0} \|A-Y\|_F^2, \min_{X \in \E_1} \|A-X\|_F^2  \right\}.
\end{equation}
Defining $\E_1$ as the closure of the set of matrices described by \eqref{eq:normalparam}, or equivalently possibly allowing $\gamma=0$, is necessary for the existence of a matrix in $\E_1$ nearest to $A$, but it also implies that $\E_0 \cap \E_1 \neq \varnothing$, because when $\gamma=0$ \eqref{eq:normalparam} becomes symmetric. However, this is not problematic, because \eqref{eq:notapartition} remains valid even if $\E_0$ and $\E_1$ are not a partition of real normal matrices.

Computing the distance from $A$ to $\E_0$ is trivial: The real symmetric matrix closest to $A$ is $(A+A^T)/2=S$, and it lies at squared distance $2\|k\|_2^2$ from $A$.
To compute instead the distance from $A$ to $\E_1$, we must optimize over the parameters of
\eqref{eq:normalparam} (possibly allowing $\gamma=0$). To this goal, observe that symmetric and skew-symmetric matrices are orthogonal with respect to the Frobenius inner product.
Therefore, we can optimize separately over the symmetric and skew-symmetric parts of \eqref{eq:normalparam}. Writing $a=\alpha-\mu$ and
$p=q^TTq$, and using $\tr T=0$, $\|qq^T\|_F^2=1$, and
$\ip{T}{qq^T}=p$, we obtain
\[
 \|S-\alpha I_3-\beta qq^T\|_F^2
 =\|T\|_F^2+3a^2+2a\beta+\beta^2-2\beta p.
\]
The latter expression is quadratic in $a$ and $\beta$, and it is thus elementary to minimize it. By performing the corresponding calculations, we see that the minimum is attained at $a=-p/2$ and $\beta=3p/2$, and the minimum value is
\[ \| T \|_F^2 - \frac32 p^2 = \| T \|_F^2 - \frac32 (q^T T q)^2.  \]
Similarly,
\[
 \|\mathcal{K}(k)-\gamma \mathcal{K}(q)\|_F^2
 =2\|k-\gamma q\|_2^2
 =2\|k\|_2^2-2(k^Tq)^2+2(\gamma-k^Tq)^2,
\]
which is minimized when $\gamma=k^Tq$, with value $2 \|k\|_2^2-2(k^T q)^2$. Hence, the squared distance from $A$ to $\E_1$ is $2\|k\|_2^2 + \|T\|_F^2 -\xi$, and this concludes the proof.
\end{proof}

Now that we have a formula for the real distance, we construct in Theorem \ref{thm:normalexample} an example of an input that admits a nonreal normal matrix at a distance closer than $d_\R$.

\begin{theorem}\label{thm:normalexample}
Let
\[
A_0=\begin{bmatrix}0&60&0\\60&0&-40\\0&40&0\end{bmatrix}.
\]
Then,
\[
d_{\C}(A_0)^2\leq3160<3200=d_{\R}(A_0)^2.
\]
In particular, no complex normal matrix nearest to $A_0$ is real.
\end{theorem}
\begin{proof}
In the notation of Proposition \ref{prop:normaldistance}, for the input $A_0$ we have
\[T=(A_0+A_0^T)/2=\begin{bmatrix}
0&60&0\\
60&0&0\\
0&0&0
\end{bmatrix}
\qquad \mathrm{and} \qquad k=\begin{bmatrix}
-40\\
0\\
0
\end{bmatrix}.\]
Clearly, the unique symmetric matrix nearest to $A_0$ is $T$, with squared distance $2\|k\|_2^2=3200$. Moreover, still using the notation of Proposition \ref{prop:normaldistance}, $\|T\|_F^2-\xi$ is the minimum, over $\left\|\begin{bmatrix}
    q_1\\
    q_2\\
    q_3
\end{bmatrix}\right\|_2=1$, of the expression $800 \cdot (9-27q_1^2q_2^2-4q_1^2)$. Adding $2\|k\|_2^2=800 \cdot 4$, denoting $\theta:=q_1^2 \in [0,1]$ and taking into account that $q_2^2 \leq 1-\theta$, we deduce that every nonsymmetric real normal matrix $N$ parametrized as in \eqref{eq:normalparam} satisfies
\[ \|A_0-N\|_F^2 \geq 800 \cdot (27 \theta^2-31 \theta+13) \geq \frac{88600}{27} \gtrsim 3281 >3200,   \]
where the second inequality is obtained by observing that the quadratic expression in $\theta$ is minimized at $\theta=31/54$. We conclude that $d_\R(A_0)^2=3200.$ It now suffices to exhibit a complex normal matrix closer than $\sqrt{3200}$ to $A_0$. One is
\[N_0=\begin{bmatrix}
-10i&60&-20i\\
60&-7i&-9\\
20i&9&17i
\end{bmatrix}.\]
Two direct calculations yield
\[N_0N_0^*=N_0^*N_0 \qquad \mathrm{and} \qquad \|A_0-N_0\|_F^2=3160.\]
\end{proof}

The size of the counterexample in Theorem \ref{thm:normalexample} is minimal. Indeed, for $n=1$ every matrix is normal and hence $\E=\C$ and the question becomes trivial, and
for $n=2$ a real normal matrix nearest to a real $A$ always exists \cite[Theorem~5.3 and Corollary~5.4]{Higham}. Note also that the proof does not guarantee that $N_0$ is a nearest complex normal matrix; in fact, this is not true and the numerical algorithm presented in \cite{NZ} finds  the nonreal normal matrix (rounded below to $5$ decimal digits): 
\[ X \approx \begin{bmatrix}
 - 13.229i  & 52.500   &  -19.843i\\
   52.500 &  - 8.2680i & -13.125\\
   19.843i  & 13.125 &  21.497i 
\end{bmatrix}   \]
which lies at the smaller squared distance $\|A_0-X\|_F^2\approx 3050$.

We now aim to construct, in Corollary \ref{cor:normaldimensions} below, counterexamples in every dimension $n \geq 3$. The basic idea is to construct a direct sum of $A_0$ as in \ref{thm:normalexample} and a multiple of the identity, say $t I_{n-3}$, but it is not obvious that this operation preserves the real distance to normal matrices (because some real normal matrices with nonzero off-diagonal blocks might lie at a smaller distance). In fact, we will need to take $t$ large enough. The first step is the technical Lemma \ref{lem:normalembedding}.

\begin{lemma}\label{lem:normalembedding}
Let $A\in\R^{n\times n}$ and $r\ge1$. For $\F =\R$ or $\C$ and for the distances to the set of normal matrices,
\[ \lim_{t\to+\infty}d_{\F}(A\oplus tI_r)=d_{\F}(A).\]
\end{lemma}

\begin{proof}
If $N$ is a  normal matrix nearest to $A$, taking $N \oplus tI_r \in \mathcal{E}$ gives the upper bound $d_\F(A\oplus tI_r) \leq d_\F(A) $. We will now show that
\[ L:=\liminf_{t \to +\infty} d_\F(A \oplus t I_r) \geq d_\F(A),  \]
implying the statement. To this goal, let $(t_k)_k \subset \R$ satisfy 
\[ t_k\to+\infty \qquad \mathrm{and} \qquad d_\F(A\oplus t_k I_r) \to L  ,\] 
and choose a sequence of normal matrices $(N_k)_k \subset \F^{(n+r) \times (n+r)}$, nearest to $A \oplus t_k I_r$ for all $k \in \N$. We next consider their block partition
\[ N_k=\begin{bmatrix}X_k&Y_k\\ Z_k&t_k I_r+W_k\end{bmatrix}.\]
Since the distance $\|N_k - (A \oplus t_k I_r)\|_F$ is bounded by $d_\F(A)$, the sequences of the blocks $X_k,Y_k,Z_k,W_k$ are all bounded. On the other hand, taking the $(1,2)$ block of
$N_kN_k^*=N_k^*N_k$,
\[  t_k(Y_k-Z_k^*)  =X_k^*Y_k+Z_k^*W_k-X_kZ_k^*-Y_kW_k^*,\]
and hence $Y_k-Z_k^*\to0$. Now consider a convergent subsequence 
\[ \left(\begin{bmatrix}
    X_{k_j} & Y_{k_j} \\ Z_{k_j} & W_{k_j}
\end{bmatrix}\right)_{k_j}  \]. The $(1,1)$ block of $N_{k_j} N_{k_j}^* = N_{k_j}^* N_{k_j}$ then gives
\[ X_{k_j}X_{k_j}^*+Y_{k_j}Y_{k_j}^*=X_{k_j}^*X_{k_j}+Z_{k_j}^*Z_{k_j}\]
which, in the limit $j \to \infty$, implies that $X$ is normal, where $X_{k_j} \to X$. Moreover,
\[  \| N_{k_j}-(A \oplus t_{k_j} I_r)\|_F \geq \|A-X_{k_j} \|_F \Rightarrow L \geq \|A-X\|_F \geq d_\F(A),  \]
where in the last step we have taken the limit for $j \to + \infty$.
\end{proof}

\begin{corollary}\label{cor:normaldimensions}
In every dimension $n\ge3$, there is a nonempty open cone of real matrices with no real nearest complex normal matrix.
\end{corollary}

\begin{proof}
For $n=3$, use Theorem~\ref{thm:normalexample}. For $n>3$, the complex distance to $A_0\oplus tI_{n-3}$ is at most $\sqrt{3160}$, whereas for $t \to +\infty$ its real distance tends to $\sqrt{3200}$ by Lemma~\ref{lem:normalembedding}. Invoking Theorem~\ref{thm:robustness} and Corollary \ref{thm:cone} completes the proof.
\end{proof}

Suppose that we randomly draw an input $A$ from the uniform distribution on the unit sphere $\|A\|_F=1$. By Corollary~\ref{cor:normaldimensions} and Corollary~\ref{thm:cone}, as long as the dimension is $n \geq 3$, there is a positive probability that $A$ does not admit any real nearest normal matrix. We next show that the probability is also strictly lower than $1$, or equivalently that the probability that $A$ has a real nearest normal matrix is also positive. In essence, Theorem \ref{thm:normalpositive} is a corollary of a sufficient condition due to A. Ruhe \cite{Ruhe} (see also \cite[Section~5]{Higham}) for a normal matrix to be one that is nearest to $A$. The criterion says that, if $N$ is normal and there exists a Hermitian matrix $H$ satisfying (i) $\gamma:=\lambda_{\max}(H)-\lambda_{\min}(H) \leq 1$ (ii) $N-A=NH-HN$, then $N \in \PP(A)$ for the nearest normal matrix problem, i.e., $N$ is a normal matrix nearest to $A$. The proof is easy; if $Y$ is normal and $\Delta=Y-N$, then some algebraic manipulations reveal that
\[ \| A - Y\|_F^2 - \|A-N\|_F^2 = \|\Delta\|_F^2 -\tr(H(\Delta^* \Delta-\Delta\Delta^*)) \geq (1-\gamma) \| \Delta \|_F^2 \geq 0.     \]
Note also that, if we strengthen the assumption to $\gamma < 1$, the same argument also shows that $N$ is the unique normal matrix nearest to $A$.
\begin{theorem}\label{thm:normalpositive}
Let $A\in\R^{n\times n}$, with $n\geq2$, and write
\[ S=\frac{A+A^T}{2},\qquad K=\frac{A-A^T}{2}.\]
Suppose that $S$ has distinct eigenvalues $\{ \lambda_i(S)\}_{i=1}^n$, and define
\[ \delta:=\min_{i\neq j}|\lambda_i(S)-\lambda_j(S)|.\]
If $\|K\|_F\leq \frac{\delta}{\sqrt{2}}$, then $S \in \PP(A)$ is a nearest normal matrix to $A$ over $\C^{n\times n}$. If the inequality is strict, then $\PP(A)=\{ S \}$, i.e., $S$ is the unique minimizer.
\end{theorem}

\begin{proof}
With no loss of generality up to applying an orthogonal similarity to $A,S,K$, we may assume that $S$ is diagonal so that $\lambda_i(S)=S_{ii}$ for $i=1,\dots,n$. Define the real symmetric matrix $H$ as
\[   H_{ij} = \begin{cases}
    0 \ &\mathrm{if} \ i=j;\\
    \frac{K_{ij}}{S_{jj}-S_{ii}} \ &\mathrm{if} \ i\neq j.
\end{cases}  \]
Then $HS-SH=K$ and $\|H\|_F\leq\|K\|_F/\delta$. Denote the largest and smallest eigenvalues of $H$ by, respectively, $\alpha:=\lambda_{\max}(H)$ and $\beta:=\lambda_{\min}(H)$ and let $\gamma:=\alpha-\beta$. Hence, $\gamma \leq \sqrt{2} \|H\|_F \leq \sqrt{2} \|K\|_F/\delta \leq 1$. The statement then follows by applying Ruhe's sufficient criterion \cite{Higham,Ruhe}.
\end{proof}

By Theorem \ref{thm:normalpositive} there exists a nonempty open subset of $\R^{n\times n}$ of inputs admitting a unique nearest normal matrix which is  real; for example it suffices to take a sufficiently small open ball centered around a real diagonal matrix with distinct eigenvalues. If we normalize such diagonal matrix to have Frobenius norm $1$, the intersection of the small ball with the unit sphere is therefore nonempty and relatively open, and thus has positive measure. Combining this observation with Corollary~\ref{cor:normaldimensions} and Corollary~\ref{thm:cone}, we obtain, for every $n \geq 3$ and for an input $A \in \R^{n \times n}$ sampled uniformly from the sphere $\|A\|_F=1$,
\[ 0<\mathbb P\bigl(\dR(A)=\dC(A)\bigr)<1.\]

\section{Nearest Hurwitz stable matrices}\label{sec:stable}

Letting $\C_- :=\{z\in\C:\Re z\le0\}$ and denoting by $\Lambda(X)$  the set of the eigenvalues of the square matrix $X$, in this section we consider the target set
\begin{equation}\label{eq:Hurwitzstable}
    \E=\{X\in\C^{n\times n}:\Lambda(X)\subseteq\C_-\}.
\end{equation}
It is clear that this set is closed and satisfies \eqref{eq:Eprop}. We follow \cite{NP} and define Hurwitz stability using the closed left half plane, i.e., we say that a matrix is Hurwitz stable if it belongs to \eqref{eq:Hurwitzstable}; in particular, matrices with defective eigenvalues on the imaginary axis are included.

The main aim of this section is to disprove that every real $A$ has a real matrix in the set $\PP(A) \subseteq \E$ of stable matrices nearest to $A$. We will give a family of counterexamples, relying on a very simple idea. If $B\in\E\cap\R^{n\times n}$, then
\begin{equation}\label{eq:detsign}
 (-1)^n\det B\ge0.
\end{equation}
Indeed, the nonreal eigenvalues of a real Hurwitz stable matrix occur in conjugate pairs with positive product, and the real eigenvalues are nonpositive; since the number of real eigenvalues has the same parity as $n$, \eqref{eq:detsign} follows. An immediate consequence of this observation is that, if $A$ is a real matrix such that its determinant has the opposite sign, a singular matrix must lie on the segment between $A$ and $B$.

The main result in  this section is Theorem \ref{thm:stable}, and we state and prove it for distances induced by an arbitrary unitarily invariant norm on $\C^{n\times n}$. Let us introduce some notation to this goal: Throughout this section, $\nu$ denotes a unitarily invariant norm on $\C^{n \times n}$, and we may also write (with slight abuse of notation) $\nu(X):=\nu(X \oplus 0_{k \times k})$ for a square matrix $X \in \C^{(n-k)\times (n-k)}$. We also define the positive constant $c_\nu= \nu( e_1e_1^T)$, and we denote by $d_{\R,\nu}$ (resp. $d_{\C,\nu}$) the functions that associate a real matrix $A$ to its real (resp. complex) distance to $\E$, where the distance is the one induced by $\nu$.
Below in this section, we will freely use the standard facts 
\begin{equation}\label{eq:uinorm}
 \nu(X)\ge c_\nu\|X\|_2,
 \qquad
 \nu(X)=c_\nu\|X\|_F\quad\text{if }\rank X=1,
\end{equation}
where $\|\cdot\|_2$ is the spectral norm. The (in)equalities in \eqref{eq:uinorm} follow by the fact that a unitarily invariant norm is always a symmetric gauge function of the singular values \cite[Theorem 3.5.18]{HJtopics}.

\begin{theorem}\label{thm:stable}
For every $m\geq 4$ and $n\geq 3$, let
\[
 A_m=\begin{bmatrix}
 -m^4&m^2&0\\0&0&m^2\\m^6&1&0
 \end{bmatrix},
 \qquad
 A_{m,n}=A_m\oplus(-m^2I_{n-3}) \in \R^{n \times n}.\]

Then, for every unitarily invariant norm $\nu$,

 \[d_{\R,\nu}(A_{m,n})=c_\nu m^2,
 \qquad and \qquad 
 0<d_{\C,\nu}(A_{m,n})\le c_\nu\left(1+\sqrt{8m^2+9}\right)\]

so that
\[
 \frac{d_{\R,\nu}(A_{m,n})}{d_{\C,\nu}(A_{m,n})}
 \ge\frac{m^2}{1+\sqrt{8 m^2+9}}\longrightarrow+\infty.
\]
In particular, no complex stable matrix nearest to $A_{m,n}$ is real.
\end{theorem}

\begin{proof}
The columns of $A_m$ are orthogonal, and
\[ A_m^TA_m= \begin{bmatrix}
     m^8+m^{12}&0&0\\
     0&m^4+1&0\\
     0&0&m^4
 \end{bmatrix}.   \]
Thus $\sigma_{\min}(A_{m,n})=m^2$. Moreover, $\det A_m=m^6(m^4+1)>0$, and therefore $(-1)^n\det A_{m,n}<0$. Let now $B$ be a real Hurwitz stable matrix. Then, by the continuity of the determinant and using \eqref{eq:detsign}, on the segment between $A_{m,n}$ and $B$ there exists a singular matrix $D=sA_{m,n}+(1-s)B$, where $0 \leq s < 1$. Let $v$ be a unit vector in $\ker D$. Then,
\[ \|A_{m,n}-D\|_2\ge\|(A_{m,n}-D)v\|_2 =\|A_{m,n}v\|_2\ge m^2.\]
Using the homogeneity of norms and \eqref{eq:uinorm}, we get
\[ \nu(A_{m,n}-B)\ge(1-s)\nu(A_{m,n}-B)=\nu(A_{m,n}-D)\ge c_\nu m^2.\]
Moreover, equality is attained at $B_m\oplus(-m^2I_{n-3})$, where
\[ B_m=\begin{bmatrix}
 -m^4&m^2&0\\0&0&0\\m^6&1&0
 \end{bmatrix}.\]
Indeed, $\det(zI_3-B_m)=z^2(z+m^4)$ and $A_m-B_m=m^2e_2e_3^T$. This shows that $d_{\R,\nu}(A_{m,n})=c_\nu m^2$.

We now exhibit a nonreal stable matrix closer to $A_{m,n}$ than $d_{\R,\nu}(A_{m,n})$. Consider
\[ H_m=\begin{bmatrix}
 -m^4&m^2+2mi&2mi-3\\
 0&0&m^2\\m^6&0&0
 \end{bmatrix}.\]
By an explicit calculation, we see that
\[ \det(zI_3-H_m)=(z-im^3)^2(z+m^4+2im^3).\]
Hence, the eigenvalues of $H_m$ have nonpositive real part and thus $H_m$ is Hurwitz stable. Moreover, a straightforward computation yields $\|A_m-H_m\|_F^2 = 8m^2+10$. For an arbitrary unitarily invariant norm we apply 
 \eqref{eq:uinorm} twice, together with the triangle inequality, to obtain the bound
 \[  \nu(A_m-H_m) \leq \nu\left( \begin{bmatrix}
   0&-2mi&3-2mi\\
   0&0&0\\
   0&0&0
 \end{bmatrix} \right) + \nu\left( \begin{bmatrix}
     0&0&0\\
     0&0&0\\
     0&1&0
 \end{bmatrix} \right) = c_\nu \left(1+\sqrt{8m^2+9}\right).   \]
To conclude the proof, we observe that $m \geq 4 \Rightarrow m^4 > 11 m^2 > 10m^2+8 \Rightarrow (m^2-1)^2 > 8m^2+9,$ and thus $H_m \oplus (-m^2 I_{n-3})$ is closer to $A_{m,n}$ than $d_{\R,\nu}(A_{m,n})$.
\end{proof}

\begin{remark}
    Since all norms on a finite dimensional vector space are equivalent, the ratio $d_\R(A_{m,n})/d_\C(A_{m,n})$ must diverge also when measured in a norm that is not unitarily invariant. However, the proof of the explicit lower bound in Theorem \ref{thm:stable} requires unitary invariance. We also observe that the nearest real Hurwitz matrix is always a stationary point of the Frobenius distance among complex Hurwitz stable matrices \cite[Theorem 7.2]{NP}. Hence, Theorem \ref{thm:stable} implies that there exist stationary points at an unbounded relatively larger distance with respect to the global minimum.
\end{remark}

Similarly to the proof of Theorem \ref{thm:normalexample}, also the proof of Theorem \ref{thm:stable} does not establish that $H_m \oplus (-m^2 I_{n-3})$ is necessarily a complex minimizer of the distance. And indeed, this is not the case. For example, for $n=3$ and $m=4$ the algorithm of \cite{NP} finds the nonreal Hurwitz stable matrix (rounded below to 5 decimal digits):

\[ X \approx 10^3 \cdot \begin{bmatrix}
    -0.2560 + 0.0001i  & 0.0134 + 0.0022i & -0.0019 + 0.0065i\\
  -0.0000 + 0.0000i & -0.0013 - 0.0007i  & 0.0134 + 0.0022i\\
   4.0960 + 0.0000i &  0.0008 + 0.0002i & -0.0001 + 0.0005i
\end{bmatrix}, \]
which lies at a Frobenius distance $\|A_{4,3}-X\|_F \approx  8.4884 < 11.747 \approx \|A_{4,3}-H_4\|_F$.

Another analogy with the nearest normal problem is that, also for the nearest Hurwitz stable problem, there is no counterexample in dimension $n < 3$. To see why, let us specialize our analysis to the Frobenius norm. For $n=1$, it is clear if $a \leq 0$ then the nearest Hurwitz stable number is $a$ itself, and if $a > 0$ then the nearest Hurwitz stable number is $0$. The case of $n=2$ is much more involved, and we treat it in Proposition \ref{prop:stablen2}.

\begin{lemma}\label{lem:lb}
    For every $M \in \C^{2 \times 2}$, it holds $| \tr(M)| \leq \sqrt{2} \|M\|_F$.
\end{lemma}
\begin{proof}
    By the Cauchy-Schwarz inequality,
    \[ |M_{11}+M_{22}| \leq  \left\| \begin{bmatrix}
        1\\
        1
    \end{bmatrix}\right\|_2 \cdot \left\| \begin{bmatrix}
        M_{11}\\
        M_{22}
    \end{bmatrix}\right\|_2 \leq \sqrt{2} \|M\|_F.   \]
\end{proof}

\begin{proposition}\label{prop:stablen2}
    Let $\E \subset \C^{2 \times 2}$ be the set of Hurwitz stable matrices, and let $\dR$ and $\dC$ be as in \eqref{eq:distances}. Then, for every $A \in \R^{2 \times 2}$, it holds $ \dR(A)=\dC(A)$.
\end{proposition}

\begin{proof}
    We may assume $A \not\in \E$, for otherwise there is nothing to prove. With no loss of generality (see \cite{NP}), we may apply an orthogonal similarity to $A$ that diagonalizes its symmetric part. Since there is freedom in choosing the sign of the orthonormal eigenvectors, this implies that we may assume
\[ A=t I_2 + \begin{bmatrix}h&k\\-k&-h\end{bmatrix}, \qquad t\in\R,\quad h,k\geq0.\]

Suppose first that $h\leq k$. Then, the eigenvalues of $A$ are $t\pm i\sqrt{k^2-h^2}$, and hence $t>0$ because $A$ is not stable. The candidate real minimizer $B=A-tI_2$ in \cite[Lemma~5.5]{NP} is stable and satisfies $\|A-B\|_F=\sqrt{2}\,t$. On the other hand, for every complex Hurwitz stable matrix $X \in \E$, it holds
\[ \|A-X\|_F \geq \frac{ |\tr(A-X)|}{\sqrt{2}}  \geq \frac{\Re\tr(A-X)}{\sqrt{2}} \geq \sqrt{2}t,\]
having used Lemma \ref{lem:lb} for the first inequality and the fact that $\Re\tr X\leq0$ for the last inequality. Hence, $B$ is also a complex minimizer.

Suppose now that $h>k$. In this case, $A-tI_2$ has the positive eigenvalue $\sqrt{h^2-k^2}$ and it is therefore not stable. The remaining candidates in \cite[Lemma~5.5]{NP} are singular and have real eigenvalues; in particular, all of them have a real upper triangular Schur form. For a unitary matrix $U$, write $M=U^*AU$ and define
\begin{equation}\label{eq:phiU}
    \phi(U):=|M_{21}|^2+[(\Re M_{11})_+]^2+[(\Re M_{22})_+]^2,
\end{equation}
where $a_+:=\max\{a,0\}.$ By \cite[Lemma~3.1]{NP}, $\phi(U)$ in \eqref{eq:phiU} is the squared distance from $M$ to the set of upper triangular Hurwitz stable matrices. Therefore, following the arguments in \cite{NP}, it must be that
\[ \dC(A)^2=\min_{U^*U=I_2}\phi(U), \quad \mathrm{and} \quad  \dR(A)^2=\min_{\substack{U^TU=I_2\\U\in\R^{2\times2}}}\phi(U).\]

Define now $U e_1 =:\begin{bmatrix}
    u_1\\
    u_2
\end{bmatrix}$,  \(x:=-2\Re(\overline{u_1}u_2)\), 
$ y:=2\Im(\overline{u_1}u_2)$, and  $z:=|u_1|^2-|u_2|^2.$ Some tedious, but straightforward, calculations show that $x^2+y^2+z^2=1$, and that
\[ \phi(U)=(h-kx)^2+(k^2-h^2)z^2+(t+hz)_+^2+(t-hz)_+^2 : = \varphi(x,z).\]
When $U$ is a unitary matrix, then the vector $\begin{bmatrix}
    x\\
    z
\end{bmatrix}$ ranges over the closed unit disk; and when $U$ is a real orthogonal matrix, the same vector ranges over the unit circle. Thus,
\[ \dC(A)^2=\min_{x^2+z^2\leq1}\varphi(x,z) \leq  \min_{x^2+z^2=1}\varphi(x,z) = \dR(A)^2.\]
Suppose now for a contradiction that $\dC(A) < \dR(A)$. Then, the minimum of $\varphi$ over the disk is not attained at the boundary. Thus, there exists an internal point $(x_0,z_0)$ such that $x_0^2+z_0^2 < 1$ and $\varphi(x_0,z_0) < \varphi(w_0,z_0)$, where $w_0:=\sqrt{1-z_0^2} > x_0$. However,
\[ \varphi(x_0,z_0)-\varphi(w_0,z_0) = (h-k x_0)^2 - \left(h-k w_0\right)^2 = k^2 (x_0^2 - w_0^2) + 2 hk (w_0-x_0)  \]
and hence, recalling $h>k\geq 0$,
\[ \varphi(x_0,z_0)-\varphi(w_0,z_0) \geq k^2 (w_0-x_0) \left(2-x_0-w_0\right) \geq 0,  \]
which is absurd. Therefore, $\dC(A) = \dR(A)$.
\end{proof}

Suppose now that we uniformly sample $A$ from the unit sphere $\{ A \in \R^{n \times n} : \|A\|_F=1 \}$, where $n \geq 3$, and consider the probability that no Hurwitz stable matrix nearest to $A$ is real. Is this probability strictly between $0$ and $1$? For Hurwitz stability, it is much more obvious that the answer is positive, for the simple reason that all matrices in a sufficiently small neighbourhood (on the unit sphere) of a Hurwitz stable matrix of unit norm, say $-I_n/\sqrt{n}$, are stable. Together with Theorem~\ref{thm:stable} and Corollary~\ref{thm:cone}, this gives
\[ 0<\mathbb P\bigl(\dR(A)=\dC(A)\bigr)<1\]
for every $n \geq 3$ and for the uniform distribution of $A$ on the unit sphere.

The question becomes more interesting if we restrict ourselves to unstable inputs; in \cite{NP}, V. Noferini and F. Poloni gave ample numerical evidence that the answer should still be positive, but a formal proof has thus far been (to our knowledge) lacking. Theorem \ref{thm:stableball} fills this gap by exhibiting an open set of Hurwitz unstable matrices that have a real nearest Hurwitz stable matrix.

\begin{theorem}\label{thm:stableball}
For $n\geq2$, let
\[
 D_n=1 \oplus -2 I_{n-1},\qquad
 \mathcal U_n=
 \left\{X\in\R^{n\times n}:\|X-D_n\|_F<\frac14\right\}.
\]
Then, for every $X\in\mathcal U_n$:
\begin{itemize}
    \item $X$ is Hurwitz unstable and, for every unitarily invariant norm $\nu$, it holds
\[  d_{\R,\nu}(X)=d_{\C,\nu}(X)=c_\nu\sigma_{\min}(X);\]
\item the Hurwitz stable matrix nearest to $X$ in the Frobenius norm is unique.
\end{itemize}
\end{theorem}

\begin{proof}
Fix $X\in\mathcal U_n$, and denote its singular values by $\sigma_1\geq\cdots\geq\sigma_n$. Taking into account that $\|X-D_n\|_2<1/4$, Weyl's inequality \cite[Theorem 3.3.16(c)]{HJtopics} guarantees that the smallest two singular values of $X$ are well separated:
\begin{equation}\label{eq:squarediffsv}
     \frac34<\sigma_n<\frac54 < \frac74 < \sigma_{n-1} \Rightarrow \sigma_{n-1}^2-\sigma_n^2 > \frac32.
\end{equation}
Moreover, $D_n=D_n^T$, and hence
\begin{equation}\label{eq:X-XT}
    \|X-X^T\|_2
 =\|(X-D_n)-(X-D_n)^T\|_2
 \leq 2\|X-D_n\|_2<\frac12.
\end{equation}
Let $u\in\R^n$ (resp. $v \in \R^n$) be a unit left (resp. right) singular vector of $X$ associated with $\sigma_n$, with $Xv=\sigma_nu$. The matrix $B=X-\sigma_nuv^T$ is singular and, by \eqref{eq:squarediffsv} and the definition of $\mathcal{U}_n$, $\|B-D_n\|_2\leq\|X-D_n\|_2+\sigma_n< 3/2$. Thus, by \cite[Theorem~5.3]{Ipsen}, the eigenvalues of $B$ can be labeled so that they satisfy
\[  |\lambda_1(B)-1|<\frac32,\qquad |\lambda_j(B)+2|<\frac32,\quad j=2,\ldots,n.\]
Since $B$ is singular, this implies that $\lambda_1(B)=0$ and $\Re(\lambda_j(B)) < 0$ for $j=2,\dots,n$. Hence, $B$ is stable. On the other hand, $\|X-D_n\|_2 < 1/4$ by assumption, and repeating the very same argument as above shows that $X$ is unstable, because one of its eigenvalues has real part greater than $3/4$.

It remains to prove that there is no Hurwitz stable matrix nearer to $X$ than $B$. To this goal, let $Y \in \C^{n \times n}$ be Hurwitz stable. By the continuity of the spectral abscissa (the maximal real part of the eigenvalues; see \cite{BLO}), there exists a positive real number $t \leq 1$ such that $X+t(Y-X)$ has a pure imaginary eigenvalue $i\omega$, $\omega\in\R$. Pick an associated eigenvector $p+iq$, with $\|p+iq\|_2=1$ and $p,q\in\R^n$. Without loss of generality up to changing the phase of this eigenvector, we may furthermore assume that $v^T q = 0$, which implies $\|X q\|_2 \geq \sigma_{n-1} \| q \|_2$. Hence,
\begin{align*}
 \|(X-i\omega I_n)(p+iq)\|_2^2  =\|Xp\|_2^2+\|Xq\|_2^2+\omega^2   -2\omega p^T(X-X^T)q\\
 \geq \sigma_n^2+\frac32\|q\|_2^2+\omega^2 -|\omega|\|p\|_2\|q\|_2\geq \sigma_n^2+\|q\|_2^2+\frac12\omega^2 \geq \sigma_n^2.
\end{align*}
Here, for the first inequality we used the Cauchy-Schwarz inequality, the relation $\|p\|_2^2+\|q\|_2^2=1$, \eqref{eq:squarediffsv}, and \eqref{eq:X-XT}. For the second inequality, we have used $\|p\|_2\leq1$ and the bound (valid for all $a,b\in\R$) $2ab\leq a^2+b^2$. Therefore, rearranging the eigenvector equation as $(X-i \omega I_n)(p+iq) = t(X-Y)(p+iq)$, we get
\[ \sigma_n  \leq\|(X-i\omega I_n)(p+iq)\|_2  =t\|(X-Y)(p+iq)\|_2  \leq\|X-Y\|_2.\]
By \eqref{eq:uinorm}, the last inequality implies $\nu(X-Y)\geq c_\nu\sigma_n$. Furthermore, $\rank(X-B)=\rank(\sigma_nuv^T)=1$. Hence, again by \eqref{eq:uinorm}, $\nu(X-B)=c_\nu\sigma_n \leq \nu(X-Y)$. Since $Y$ was an arbitrary stable matrix, this proves that $B$ is a minimizer of the distance to $X$.

It remains to prove that, for every $X$ and for $\nu(\cdot)=\| \cdot \|_F$, $B$ is the unique minimizer. To this goal, observe that if $Y$ is a Hurwitz stable matrix nearest to $X$ in the Frobenius distance, then $\|X-Y\|_F=\sigma_n$, and hence the estimates obtained above yield
\[  \sigma_n^2 \leq \sigma_n^2 + \|q\|_2^2 +\frac{\omega^2}{2} \leq t^2 \|X-Y\|_2^2 \leq t^2 \sigma_n^2 \leq \sigma_n^2.  \]
Hence, the inequality chain above must hold with equality, that is, $t=1$, $q=0$, and $\omega=0$. In particular, $Y$ is a singular matrix at distance $\sigma_n$ from $X$. Since $\sigma_{n-1}>\sigma_n$, we conclude $Y=B$, proving uniqueness.
\end{proof}

In the notation of Theorem \ref{thm:stableball}, if we scale the elements of $\mathcal{U}_n$ by $\|D_n\|_F^{-1}$ we obtain an open ball centered at $D_n/\|D_n\|_F$, whose elements are still real unstable matrices having a nearest Hurwitz stable matrix which is real. Again, the intersection of this ball with the unit sphere has positive measure, and thus (taking also into account Corollary \ref{thm:cone} and Theorem \ref{thm:stable}) we are able to conclude that, for $n\geq3$ and $A$ uniformly drawn from the unit sphere,

\[ 0<\mathbb P\bigl(\dR(A)=\dC(A) \mid A \not\in \E\bigr)<1.\]

\section{Method to construct the counterexamples}\label{sec:method}

To keep the proofs of Theorem \ref{thm:normalexample} and Theorem \ref{thm:stable} concise, we did not include any intuition of why the matrix pairs $(A_0,N_0)$ and $(A_m,H_m)$ successfully did their job. The purpose of this section is to provide a more detailed account of how the counterexamples were constructed.

\begin{itemize}
    \item For Theorem \ref{thm:normalexample}, our guesses were informed by the algorithm in \cite{NZ}. We first constructed a sparse input $A_0$ by reverse-engineering the goal of computing  $d_\R$ exactly. (This required a few attempts.) Then, for $A_0$ as in the statement, we noticed that the numerical output of our algorithm had the approximate pattern $\begin{bmatrix}
    ai & b & -ci\\
    b & di & -e\\
    ci & e & fi
\end{bmatrix}$ for real $a,b,c,d,e,f$. Within this family, normality is equivalent to the equations $be = c(f-a)$ and $ce= b(d-a)$. Then, by some na\"{i}ve integer optimization in MATLAB, we looked for a ``nice-looking" normal matrix $N_0$ in this family and close enough to $A_0$ to allow us to conclude the proof via exact computations.
\item For Theorem \ref{thm:stable}, we constructed the family $A_m$ to guarantee three properties that, together, determine $d_\R$ exactly: a determinant of the ``wrong" sign, an easily computable smallest singular value, and a nearest singular matrix $B_m$ that was also Hurwitz stable. To complete the proof, we needed a complex Hurwitz stable $H_m$ closer than $B_m$. We searched for one in the ansatz family $H_m = A_m + \begin{bmatrix} 0 & a & b\\0 &0  & 0\\0 & c & 0 \end{bmatrix}$ where $a,b,c \in \C(m)$, i.e., rational functions in $m$. As explained in \cite{NP,NP2}, it is typical for nearest Hurwitz stable matrices in the Frobenius norm to have Jordan blocks associated with imaginary eigenvalues. Since $A_m$ has one $O(m^4)$ eigenvalue and two $O(m^3)$ eigenvalues, we imposed via MATLAB symbolic computations that $i m^3$ is a double eigenvalue of $H_m$. This gave the symbolic solution $\begin{bmatrix}a\\b\end{bmatrix} = m^{-4}\begin{bmatrix}m^2(2m^3i-c-1 )\\2m^5i-3m^4-c-1\end{bmatrix}$; picking $c=-1$, the expressions for $a$ and $b$ became polynomials with Gaussian integer coefficients, which was a nice feature for exposition purposes. Finding this example was sufficient to prove that the distance ratio $d_\R/d_\C$ diverges as $m \to \infty$, and therefore we did not attempt further improvements. We do not have a proof of optimality of the $O(m)$ behaviour of the resulting lower bound of the ratio. However, we do not have strong evidence of non-optimality either; in particular, we note that the $3 \times 3$ matrix $X$ found (for $m=4$) by the algorithm in \cite{NP}, has a double eigenvalue $\lambda \approx 51.67i$ which is not terribly far from $i m^3 = 64i$.
\end{itemize}

\section*{AI statement}

The ideas and proofs in the manuscript were provided by the human authors. We used ChatGPT 6.0 for proof and exposition auditing, and it only gave minor suggestions.

\end{document}